\documentclass[10pt]{article}
\usepackage{graphicx}
\usepackage{amssymb}
\usepackage{color,graphicx}
\usepackage{amsmath,amsthm,amscd}
\usepackage[latin1]{inputenc}
\usepackage{fancyhdr}
\usepackage{pb-diagram}
\numberwithin{equation}{section}
\newtheorem{theorem}{Theorem}[section]

\newtheorem{proposition}[theorem]{Proposition}
\newtheorem{lemma}[theorem]{Lemma}

\newtheorem{corollary}[theorem]{Corollary}

\newtheorem{definition}[theorem]{Definition}
\theoremstyle{definition}

\newtheorem{remark}[theorem]{Remark}

\newcommand{\ol}[1]{\overline{#1}}

\newcommand{\too}{\longrightarrow}

\newcommand{\xto}[1]{\xrightarrow{#1}}

\newcommand{\HH}{\ensuremath{\mathcal{H}}}

\newcommand{\mm}{\ensuremath{\mathfrak{m}}}

\newcommand{\MM}{\ensuremath{\mathfrak{M}}}

\newcommand{\pp}{\ensuremath{\mathfrak{P}}}

\newcommand{\fg}{\MM}%%%%%% módulos finitamente generados (f.g)
\newcommand{\pry}{\pp}%%%%%% módulos proyectivos f.g

\DeclareMathOperator{\PF}{\mathcal{PF}}%
\DeclareMathOperator{\PSF}{\mathcal{PSF}}%

\DeclareMathOperator{\GL}{GL}%
\DeclareMathOperator{\Aut}{Aut}%
\DeclareMathOperator{\Max}{Max}%

\newfont{\fuentea}{cmsy10 at 16pt}
\newfont{\fuenteb}{cmsy10 at 10pt}
\DeclareFixedFont{\zcal}{OT1}{pzc}{m}{sl}{16pt}%
\DeclareFixedFont{\titulo}{OT1}{pbk}{bx}{sc}{30pt}%
\DeclareFixedFont{\trm}{OT1}{ptm}{b}{sc}{12pt}%
\DeclareFixedFont{\trw}{OT1}{phv}{b}{bx}{14pt}%

\def\@roman#1{\romannumeral #1}
\makeatother

\title{\textbf{Extended modules and Ore extensions}}
\author{Vyacheslav Artamonov\\
\texttt{viacheslav.artamonov@gmail.com}\\
Moscow State University\\
Oswaldo Lezama\\
\texttt{jolezamas@unal.edu.co}\\
Universidad Nacional de Colombia\\
William Fajardo\\
\texttt{wafajardoc@unal.edu.co}\\
Universidad Nacional de Colombia}
\date{}
\begin{document}
\maketitle
\begin{abstract}
\noindent In this paper we investigate extended modules for a special class of Ore extensions. We
will assume that $R$ is a ring and $A$ will denote the Ore extension $A:=R[x_1,\dots,x_n;\sigma]$
for which $\sigma$ is an automorphism of $R$, $x_ix_j=x_jx_i$ and $x_ir=\sigma(r)x_i$, for every
$1\leq i,j\leq n$. With some extra conditions over the ring $R$, we will prove Vaserstein's,
Quillen's patching,
%Roitman's,
Horrocks' and Quillen-Suslin's theorems for this type of
non-commutative rings.

\bigskip


\bigskip

\end{abstract}
%-------------------------------------------------
%-------------------------------------------------
\section{Introduction}\label{definitionexamplesspbw}

The study of finitely generated projective modules over a ring $B$ induces the notions of
$\mathcal{PSF}$, $\mathcal{PF}$, Hermite ($\mathcal{H}$), $d$-Hermite rings, and many other classes
of interesting rings. $B$ is $\mathcal{PSF}$ if any finitely generated projective left $B$-module
is stably free; we say that $B$ is $\mathcal{PF}$ if any finitely generated projective left
$B$-module is free; $B$ is Hermite ($\mathcal{H}$) if any stably free left $B$-module is free. The
ring $B$ is \textit{$d$-Hermite} if any stably free left $B$-module of rank $\geq d$ is free. Note
that $\mathcal{H} \cap  \mathcal{PSF} = \mathcal{PF}$. In this paper we will study the class of
extended modules which is also very useful for the investigation of projective modules. This
special class arises when we try to generalize the famous Quillen-Suslin theorem about projective
modules over polynomial rings with coefficients in $PIDs$ to a wider classes of coefficients, or
yet to the non-commutative rings of polynomial type (see \cite{Artamonov}, \cite{Artamonov2},
\cite{Artamonov3} and \cite{Lam}). We are interested in investigating extended modules and rings
for some special classes of Ore extensions. Thus, if nothing contrary is assumed, we will suppose
that $R$ is a ring and $A$ will denote the Ore extension $A:=R[x_1,\dots,x_n;\sigma]$ for which
$\sigma$ is an automorphism of $R$, $x_ix_j=x_jx_i$ and $x_ir=\sigma(r)x_i$, for every $1\leq
i,j\leq n$. In some places we will assume some extra conditions on $R$.

Some notation is needed as well as to recall some definitions and basic facts. $\langle X\rangle$
will denote the two-sided ideal of $A$ generated by $x_1,\dots,x_n$. Often we will used also the
following notation for $A$, $A=\sigma(R)\langle X\rangle$. An element
$p=c_0+c_1X_1+\cdots+c_tX_t\in A$, with $c_0,c_i\in R$ and $X_i\in Mon(A)$, $1\leq i\leq t$, will
be denoted also as $p=p(X)$, where ${\rm Mon}(A):= \{x_1^{\alpha_1}\cdots x_n^{\alpha_n}\mid
\alpha=(\alpha_1,\dots ,\alpha_n)\in \mathbb{N}^n\}$. The elements of $Mon(A)$ will be represented
by $x^\alpha$ or in capital letters, i.e., $x^\alpha=X$. All modules are left modules if nothing
contrary is assumed. We will use the left notation for homomorphisms and row notation for matrix
representation of homomorphisms between free modules (see Remark 1 in \cite{Gallego3}).

\begin{definition}\label{extendedmodules}
Let $T\supseteq S$ be rings.
\begin{enumerate}
\item[\rm (i)]Let $M$ be a $T$-module, $M$ is extended from $S$ if there exists an $S$-module $M_0$ such that $M\cong
T\otimes_S M_0$. It also says that $M$ is an extension of $M_0$ with respect to $S$.
\item[\rm (ii)]$\fg(T)$ denotes the family of finitely generated $T$-modules,
$\pry(T)$ the family of projective $T$-modules in $\fg(T)$ and $\pry^S(T)$ the family of modules in
$\pry(T)$ extended from $S$.
\item[\rm (iii)]The ring $T$ is extended with respect to $S$, also called $S$-extended, if every finitely generated
projective $T$-module is extended from $S$, i.e., $\pry(T)\subseteq\pry^{S}(T)$. For the Ore
extension $A:=R[x_1,\dots,x_n;\sigma]$, we will say that $A$ is $\mathcal{E}$ if $A$ is
$R$-extended.
\end{enumerate}
\end{definition}

\begin{proposition}[Bass's Theorem]\label{Bass_ideal}
Let $I$ be a two-sided ideal of a ring $R$ such that $I\subseteq Rad (R)$, and let $P,Q$ be
projective $R$-modules. Then, $P\cong Q$ if and only if $P/IP\cong Q/IQ$ as $R/I$-modules. In
particular, $P$ is $R$-free if and only if $P/IP$ is $R/I$-free.
\end{proposition}
\begin{proof}
See \cite{Bass1}, Lemma 2.4.
\end{proof}

\begin{proposition}\label{913}
Let $T\supseteq S\supseteq R$ be rings and $M$ a $T$-module.
\begin{enumerate}
\item[\rm (i)]If $M$ is an extension of $M_0$ with respect to $S$ and $M_0$ is an extension of $L_0$ with
respect to $R$, then $M$ is an extension of $L_0$ with respect to $R$.
\item[\rm (ii)]If $T$ is $R$-extended, then $T$ is $S$-extended.
\item[\rm (iii)] Let $I$ be a proper two-sided ideal of
$T$, with $S\cong T/I$. If $T$ is $\PF$ then $S$ is $\PF$.
\item[\rm (iv)] Let $J$ be a two-sided ideal of $R$ such that $J\subseteq Rad(R)$.
If $R/J$ is $\PF$, then $R$ is $\PF$.
\end{enumerate}
\end{proposition}
\begin{proof}
(i) $M\cong T\otimes_S M_0$, $M_0\cong S\otimes_R L_0$, then $M\cong T\otimes_S S\otimes_R L_0\cong
T\otimes_R L_0$.

(ii) $M\cong T\otimes_R M_0\cong (T\otimes_S S)\otimes_R M_0\cong T\otimes_S(S\otimes_R M_0)$.

(iii) Let $M\in\pry(S)$. Then, $M\oplus M'\cong S^r$ for some $S$-module $M'$, therefore
$(T\otimes_SM)\oplus M''\cong T^r$, i.e., $T\otimes_SM\in\pry(T)$, so $T\otimes_SM$ is $T$-free,
and hence $T\otimes_SM\cong T^{\ell}\cong T\otimes_SS^{\ell}$. Whence, $M\cong S\otimes_S
M\cong(T/I\otimes_TT)\otimes_SM\cong T/I\otimes_T(T\otimes_SM) \cong
T/I\otimes_T(T\otimes_SS^{\ell}) \cong (T/I\otimes_TT)\otimes_SS^{\ell} \cong S\otimes_SS^{\ell}
\cong S^{\ell}$. Therefore, $S$ is $\PF$.

(iv) Let $M\in\pry(R)$, then $R/J\otimes_RM\in\pry(R/J)$ so that
\[
M/JM\cong R/J\otimes_RM\cong (R/J)^n\cong R/J\otimes_RR^{n}\cong R^n/JR^n.
\]
From Proposition \ref{Bass_ideal}, $M\cong R^n$, and hence $R$ is $\PF$.
\end{proof}

\section{Extended rings and Ore extensions}

From now on in the present paper (except in the last section) we will assume that $R$ is a
commutative ring and $A$ will denote the Ore extension $A:=R[x_1,\dots,x_n;\sigma]$ for which
$\sigma$ is an automorphism of $R$, $x_ix_j=x_jx_i$ and $x_ir=\sigma(r)x_i$, for every $1\leq
i,j\leq n$.

\begin{proposition}\label{8.1.3}
Let $M$ be an $A$-module. Then,
\begin{enumerate}
\item[\rm (i)]If $M$ is free, then $M$ is extended from $R$.
\item[\rm (ii)]If $M$ is an extension of $M_0$ with respect to $R$, then
\begin{equation}
M_0\cong M/\langle X\rangle M.
\end{equation}
Moreover, if $M$ is finitely generated {\rm(}projective, stably free{\rm)} as $A$-module, then
$M_0$ is finitely generated {\rm(}projective, stably free{\rm)} as $R$-module.
\end{enumerate}
\end{proposition}
\begin{proof}
(i) $M\cong A^{(Y)}$, then $M\cong A\otimes_R R^{(Y)}$ (note that this property is still valid for
any couple of rings $A\supseteq R$).

(ii) First note that $A/\langle X\rangle\cong R$; if $M\cong A\otimes_R M_0$ then $A/\langle
X\rangle \otimes_{A}M\cong A/\langle X\rangle \otimes_{A}A\otimes_R M_0$, i.e., $M/\langle X\rangle
M\cong A/\langle X\rangle \otimes_R M_0\cong R\otimes_R M_0\cong M_0$.

Let $M=\langle z_1,\dots ,z_t\rangle$ and $w\in M_0$, then $w=\overline{z}$ with $z\in M$; there
exist $p_1(X),\dots ,p_t(X)\in A$ such that $w=\overline{z}=\overline{p_1(X)z_1+\cdots
+p_t(X)z_t}=p_{01}\overline{z_1}+\cdots +p_{0t}\overline{z_t}$, where $p_{0i}$ is the independent
term of $p_i(X)$, $1\leq i\leq t$. Hence, $M_0=$ $\langle \overline{z_1}, \dots
,\overline{z_t}\rangle$.

If $M$ is projective, then $M\oplus M'=A^{(Y)}$, and $A/\langle X\rangle\otimes_{A} (M\oplus
M')\cong A/\langle X\rangle\otimes_{A} A^{(Y)}$, i.e., $M_0\oplus M'/\langle X\rangle M'\cong
R^{(Y)}$.

If $M$ is stably free, then $M\oplus A^r=A^{s}$, so applying $A/\langle X\rangle\otimes_{A}$ we get
$M_0\oplus R^r\cong R^{s}$.
\end{proof}

We can give a matrix description of extended modules and rings. For this we firstly we recall the
definition of square similar matrices.

\begin{definition}
Let $S$ be a ring and $F,G\in M_s(S)$, it says that $F$ and $G$ are similar if there exists $P\in
G_s(S)$ such that $F=PGP^{-1}$. In particular, let $F(X)$ be a square matrix over $A$ of size
$s\times s$ and $F(0)$ the matrix over $R$ obtained from $F(X)$ replacing all the variables
$x_1,\dots ,x_n$ by $0$,
\begin{equation}
F(X)\approx F(0)\Leftrightarrow F(0)=P(X)F(X)P(X)^{-1},\, \text{for some}\, P(X)\in GL_s(A).
\end{equation}
\end{definition}

\begin{theorem}\label{matrixE}
Let $M$ be a finitely generated projective $A$-module and $F(X)\in M_s(A)$ be an idempotent matrix
such that $M=\langle F(X)\rangle$, where $\langle F(X)\rangle$ is the $A$-module generated by the
rows of $F(X)$.
\begin{enumerate}
\item[\rm (i)]If $F(X)\approx F(0)$, then $M$ is extended from $R$.
\item[\rm (ii)]If $M$ is extended from $R$, then there exists a non zero matrix $P(X)\in M_s(A)$ such that $P(X)F(0)=F(X)P(X)$.
\item[\rm (iii)]If $A$ is such that for every $s\geq 1$, given an idempotent matrix $F(X)\in M_s(A)$, $F(X)\approx F(0)$, then $A$ is $\mathcal{E}$.
\item[\rm (iv)]If $A$ is $\mathcal{E}$, then for each $s\geq 1$, given an idempotent matrix $F(X)\in M_s(A)$, there exists a non zero matrix
$P(X)\in M_s(A)$ such that $P(X)F(X)=F(0)P(X)$.
\end{enumerate}
\end{theorem}
\begin{proof}
(i) There exists $P(X)\in GL_s(A)$ such that $P(X)F(X)P(X)^{-1}=F(0)$. Since $A$ is
quasi-commutative, $F(0)\in M_s(R)$ is idempotent. Let $M_0:=\langle F(0)\rangle$ the $R$-module
generated by the rows of $F(0)$, then $M_0$ is a finitely generated projective $R$-module. We will
prove that $\langle F(X)\rangle\cong A\otimes_R M_0$, i.e., $M$ is extended from $R$. $F(X),P(X)$
define $A$-endomorphisms of $A^s$, with $P(X)$ bijective, and $F(0)$ define a $R$-endomorphism of
$R^s$. Let $G(X):=i_{A}\otimes_R F(0)$, then the following diagram
\begin{equation*}
\begin{CD}
A^{s}\cong A\otimes_R R^s @>{G(X)}>>A\otimes_R R^s \cong A^{s}\\
@V{P(X)}VV @VV{P(X)}V \\
A^{s}@>{F(X)}>>A^{s}.
\end{CD}
\end{equation*}
is commutative since $P(X)F(X)P(X)^{-1}=F(0)$ and the matrix of $G(X)$ in the canonical basis of
$A^s$ coincides with $F(0)$. From this we conclude that $\langle F(X)\rangle=Im(F(X))\cong
Im(G(X))\cong Im(i_A)\otimes Im(F(0))=A\otimes_R M_0$.

(ii) We have $M\cong A\otimes_R M_0$, for some finitely generated projective $R$-module $M_0$, but
by Proposition \ref{8.1.3}, $M_0\cong M/\langle X\rangle M=\langle F(X)\rangle/\langle
X\rangle\langle F(X)\rangle=\langle F(0)\rangle$, so $M_0$ is generated by $s$ elements. Thus, we
have $Im(F(X))=\langle F(X)\rangle\cong A\otimes_R \langle F(0)\rangle=Im(G(X))$, where $F(X)$ and
$G(X)=F(0)$ are the idempotent endomorphisms of $A^{s}$ as in (i). Let $H(X):Im(F(X))$ $\rightarrow
Im(F(0))$ be an isomorphism. We have
\begin{center}
$A^{s}=Im(F(X))\oplus \ker(F(X))=Im(F(0))\oplus \ker(F(0))$.
\end{center}
Let $T(X)$ be any $A$-homomorphism from $\ker(F(X))$ to $\ker(F(0))$, for example,
$\textbf{\emph{w}}(X)T(X):=\textbf{\emph{w}}(0)$, with $\textbf{\emph{w}}(X)\in \ker(F(X))$. We
have the following diagram
\begin{equation*}
\begin{CD}
A^{s} @>{F(X)}>>A^{s}\\
@V{P(X)}VV @VV{P(X)}V \\
A^{s}@>{F(0)}>>A^{s}
\end{CD}
\end{equation*}
where $P(X)$ is the $A$-homomorphism defined by $P(X):=H(X)\oplus T(X)$. Note that the diagram is
commutative: If $\textbf{\emph{v}}(X)=\textbf{\emph{u}}(X)F(X)+\textbf{\emph{w}}(X)$, with
$\textbf{\emph{u}}(X)\in A^{s}$ and $\textbf{\emph{w}}(X)\in \ker(F(X))$, then
\begin{center}
$\textbf{\emph{v}}(X)F(X)P(X)=\textbf{\emph{u}}(X)F(X)^2P(X)+
\textbf{\emph{w}}(X)F(X)P(X)=\textbf{\emph{u}}(X)F(X)H(X)$;
\end{center}
on the other side,
\begin{center}
$\textbf{\emph{v}}(X)P(X)F(0)=[\textbf{\emph{u}}(X)F(X)H(X)+\textbf{\emph{w}}(X)T(X)]F(0)=
\textbf{\emph{u}}(X)F(X)H(X)F(0)+\textbf{\emph{w}}(X)T(X)F(0)=\textbf{\emph{u}}(X)F(X)H(X)$,
\end{center}
where the last equality follows from the fact that $\textbf{\emph{u}}(X)F(X)H(X)\in Im(F(0))$ and
$\textbf{\emph{w}}(X)T(X)\in \ker(F(0))$.

This proves that $P(X)F(0)=F(X)P(X)$. If $F(X)\neq 0$, then $H(X)\neq 0$ and hence $P(X)\neq 0$; if
$F(X)=0$, then $F(0)=0$, $H(X)=0$, $\ker(F(X))=A^s=\ker(F(0))$ and we can take $T(X)=P(X)=i_{A^s}$.

(iii) is a direct consequence of (i) and (iv) follows from (ii).
\end{proof}

\begin{remark}
In the proof of (ii) we observed that
\begin{center}
$\ker(F(X))\cong A^s/Im(F(X))$ and $\ker(F(0))\cong A^s/Im(F(0))$
\end{center}
are finitely presented modules with $Im(F(X))\cong Im(F(0))$. If there exists at least one
surjective homomorphism $T(X)$ from $\ker(F(X))$ to $\ker(F(0))$ and $A$ is left Noetherian, then
$P(X)$ is surjective, and hence bijective, i.e., in this situation $F(X)\approx F(0)$.
\end{remark}

The following results relate the extended condition $\mathcal{E}$ and the $\mathcal{PF}$ condition
for the Ore extension $A$.

\begin{proposition}\label{8.1.5}
Suppose that $R$ is $\mathcal{PF}$. $A$ is $\mathcal{E}$ if and only if $A$ is $\mathcal{PF}$.
\end{proposition}
\begin{proof}
$\Rightarrow)$: Let $M$ be a f.g. projective $A$-module, then $M\cong A\otimes_R M_0$, where $M_0$
is a f.g. projective $R$-module (Proposition \ref{8.1.3}). But since $R$ is $\mathcal{PF}$, then
$M_0$ is $R$-free and hence $M$ is $A$-free.

$\Leftarrow)$: If $M$ is a f.g. projective left $A$-module, then $M$ is $A$-free, then by
Proposition \ref{8.1.3}, $M$ is extended from $R$.
\end{proof}

\begin{proposition}\label{APfThenRPf}
If $A$ is $\mathcal{PF}$, then $R$ is $\PF$.
\end{proposition}
\begin{proof}
We know that $A/\langle X\rangle\cong R$, so the result follows form Proposition \ref{913}.
\end{proof}
\begin{corollary}\label{APfIffRPfAextended}
$A$ is $\PF$ if and only if $R$ is $\PF$ and $A$ is $\mathcal{E}$
\end{corollary}
\begin{proof}
This is direct consequence of Propositions \ref{8.1.5} and \ref{APfThenRPf}.
\end{proof}

\begin{corollary}\label{927}
Let $R$ be a Noetherian, regular and $\PSF$ ring. Then, $A$ is $\HH$ if and only if $R$ is
$\mathcal{PF}$ and $A$ is $\mathcal{E}$.
\end{corollary}
\begin{proof}
This follows from \cite{lezamareyes1} Corollary 2.8, and Corollary \ref{APfIffRPfAextended}.
\end{proof}

\begin{remark}\label{ejemrsigminv}
(i) Let $A:=R[x;\sigma]$ be the skew polynomial ring over $R=K[y]$, where $K$ is a field and
$\sigma(y):=y+1$. From \cite{McConnell} 12.2.11 we know that $A$ is not $\mathcal{E}$ with respect
to $R$, and hence, from numeral (ii) in Proposition \ref{913}, we conclude that $A$ is not
$\mathcal{E}$ with respect to $K$. But precisely observe that $A=K[t;i_K][x;\sigma]$ is an Ore
extension such that the $\sigma'$s are different and $tx\neq xt$. Thus, the conditions we are
assuming in this paper about the commutativity of the variables and the restriction to only one
automorphism for the ring of coefficients are more than important.

(ii) On the other hand, since $R$ is a commutative principal ideal domain ($PID$) then $R$ is
$\PF$, therefore, by Corollary \ref{APfIffRPfAextended}, $A$ is not $\PF$. This means that in
Theorem \ref{quillensuslinforore} below we can not weak the condition on $K$ to be a commutative
$PID$.

(iii) In addition, observe that $R$ is a commutative Noetherian regular ring with finite Krull
dimension and however $A$ is not $\mathcal{E}$, so the Bass-Quillen conjecture (see \cite{Lam}) in
the case of our Ore extensions conduces to Quillen-Suslin Theorem \ref{quillensuslinforore}.

(iv) Finally, this example also shows that although $R$ is $\mathcal{H}$, $R[x;\sigma]$ is not
$\mathcal{H}$. In fact, since $R$ is $\mathcal{PF}$, we conclude that $R$ is $\mathcal{PSF}$, so
the claimed follows from Corollary \ref{927}. Thus, the Hermite conjecture for Ore extensions fails
(see \cite{Lam}).
\end{remark}

\section{Varserstein's theorem}

Let $R$ be a commutative ring, the Vaserstein's theorem in commutative algebra says that if
$F(x_1,\dots,x_n)\in M_{r\times s}(R[x_1,\dots,x_n])$, then, $F(x_1,\dots,x_n)\sim F(0)$ if and
only if for every $\mm \in \Max(R)$, $\overline{F(x_1,\dots,x_n)}\sim \overline{F(0)}$, where
$\overline{F(x_1,\dots,x_n)}$ represents the image of $F(x_1,\dots,x_n)$ in
$R_{\mm}[x_1,\dots,x_n]$ and $\sim$ denotes the relation of equivalence between matrices, i.e.,
\begin{center} $F(X)=P(X)F(0)Q(X)$,
\end{center}
with $P(X)\in GL_r(R[x_1,\dots,x_n])$ and $Q(X)\in GL_s(R[x_1,\dots,x_n])$. In this section we
extends this theorem to Ore extensions of type $A:=R[x_1, \dots,x_n;\sigma]$.

Recall (see \cite{lezamaore}) that if $S$ is a multiplicative system of $R$ and $\sigma(S)\subseteq
S$, then $S^{-1}A$ exists and
\begin{center}
$S^{-1}A\cong (S^{-1}R)[x_1,\cdots x_n;\overline{\sigma}]$, with
$\overline{\sigma}(\frac{r}{s}):=\frac{\sigma(r)}{\sigma(s)}$.
\end{center}
In particular, if $\mm\in \Max(R)$ and $S:=R-\mm$, we write
\begin{center}
$A_{\mm}:=S^{-1}A\cong R_{\mm}[x_1,\dots,x_n; \overline{\sigma}]$, where $\sigma$ satisfies
$\sigma(s)\notin \mm$ for any $s\notin \mm$.
\end{center}
From now on in the present paper we will assume that $\sigma$ satisfies the following condition:
\begin{center}
$\clubsuit$: Given $\mm \in \Max(R)$, if $s\notin \mm$, then $\sigma(s)\notin \mm$.
\end{center}

Some preliminary results are needed for the main theorem.

\begin{proposition}\label{931g}
Let $B$ be a ring and $\sigma$ an endomorphism of $B$. Then,
\begin{enumerate}
\item[\rm (i)]For every $r\geq 1$, $M_r(B[x_1,\dots,x_n;\sigma])\cong M_r(B)[x_1,\dots,x_n;\sigma]$.
\item[\rm (ii)]If $\sigma(Z(B))\subseteq Z(B)$ and $s\in Z(B)$, then
\begin{align*}
\phi:B[x_1,\dots,x_n;\sigma] & \to B[x_1,\dots,x_n;\sigma]\\
p(x_1,\dots,x_n)& \mapsto p(sx_1,\dots,sx_n)
\end{align*}
is a ring homomorphism.
\item[\rm (iii)]$\varphi$ defined as
\begin{center}
$\varphi:B[x_1,\dots,x_n;\sigma]\to B[x_1,\dots,x_n;y_1,\dots,y_n;\sigma]$ ,
$\varphi(p(x_1,\dots,x_n)):=p(x_1+y_1,\dots, x_n+y_n)$
\end{center}
is a ring homomorphism.
\end{enumerate}
\end{proposition}
\begin{proof}
(i) Using an inductive argument we only need to show that $M_r(B[x_1;\sigma])\cong
M_r(B)[x_1;\sigma]$. If we define $\sigma(F):=[\sigma(f_{ij})]$, with $F:=[f_{ij}]\in M_r(B)$, then
the claimed isomorphism is given by
\begin{center}
$F^{(0)}+F^{(1)}x_1+\cdots+F^{(t)}x_1^t\mapsto [\sum_{k=0}^t f_{ij}^{(k)}x_1^k]$.
\end{center}
(ii) It is clear that $\phi$ is additive and $\phi(1)=1$. So, we have to show that $\phi(ax^\alpha
bx^\beta)=\phi(ax^\alpha)\phi(bx^\beta)$ for every $a,b\in B$ and $\alpha,\beta\in \mathbb{N}^n$.
Since $\sigma^k(s)\in Z(B)$ for every $k\geq 0$, then
\begin{align*}
\phi(ax^\alpha bx^\beta) & =
\phi(a\sigma^\alpha(b)x^{\alpha+\beta})=a\sigma^\alpha(b)(sx_1)^{\alpha_1+\beta_1}\cdots (sx_n)^{\alpha_n+\beta_n}\\
& = a\sigma^{\alpha}(b)s\sigma(s)\sigma^2(s)\cdots \sigma^{\alpha_1+\alpha_2+\cdots+\alpha_n+\beta_1+\beta_2+\cdots+\beta_n-1}(s)x^{\alpha+\beta};\\
\phi(ax^\alpha)\phi(bx^\beta) & = a(sx_1)^{\alpha_1}\cdots (sx_n)^{\alpha_n}b(sx_1)^{\beta_1}\cdots (sx_n)^{\beta_n}\\
& = a\sigma^{\alpha}(b)s\sigma(s)\sigma^2(s)\cdots
\sigma^{\alpha_1+\alpha_2+\cdots+\alpha_n+\beta_1+\beta_2+\cdots+\beta_n-1}(s)x^{\alpha+\beta}.
\end{align*}

(iii) Obviously $\varphi$ is additive and $\varphi(1)=1$. Only rest to show that $\varphi(ax^\alpha
bx^\beta)=\varphi(ax^\alpha)\varphi(bx^\beta)$ for every $a,b\in B$ and $\alpha,\beta\in
\mathbb{N}^n$.
\begin{align*}
\varphi(ax^\alpha bx^\beta) & =
\varphi(a\sigma^\alpha(b)x^{\alpha+\beta})=a\sigma^\alpha(b)(x_1+y_1)^{\alpha_1+\beta_1}\cdots (x_n+y_n)^{\alpha_n+\beta_n};\\
\varphi(ax^\alpha)\varphi(bx^\beta) & = a(x_1+y_1)^{\alpha_1}\cdots (x_n+y_n)^{\alpha_n}
b(x_1+y_1)^{\beta_1}\cdots (x_n+y_n)^{\beta_n}\\
& =a\sigma^{\alpha}(b)(x_1+y_1)^{\alpha_1+\beta_1}\cdots (x_n+y_n)^{\alpha_n+\beta_n}.
\end{align*}
\end{proof}
\begin{lemma}\label{lema1720}
Let $B$ be a ring, $S\subset Z(B)$ a multiplicative system of $B$. Let $B[x_1,\dots,x_n;\sigma]$ be
an Ore extension such that $\sigma(Z(B))\subseteq Z(B)$. Given the matrices $F(X)\in M_{r\times
s}(B[x_1,\dots,x_n;\sigma]), G(X)\in M_{s\times t}(B[x_1,\dots,x_n;\sigma])$ and $H(X)\in
M_{r\times t}(B[x_1,\dots,x_n;\sigma]), $ let $\ol{L(X)}$ be the image of the matrix $L(X)$
corresponding to the canonical homomorphism
\[
B[x_1,\dots,x_n;\sigma]\to (S^{-1}B)[x_1,\dots,x_n;\overline{\sigma}],\quad
\overline{\sigma}(\tfrac{r}{s}):=\tfrac{\sigma(r)}{\sigma(s)}.
\]
Suppose that $\ol{F(X)}\:\ol{G(X)}=\ol{H(X)}$ and $F(0)G(0)=H(0)$. Then, there exists $s\in S$ such
that $F(sX)G(sX)=H(sX)$, where $L(sX):=L(sx_1,\dots,sx_n)$.
\end{lemma}
\begin{proof}
Let $D(X):=F(X)G(X)-H(X)$; since $F(0)G(0)-H(0)=0$ then
\[
D(X)=D^{(1)}x^{\alpha_1}+D^{(2)}x^{\alpha_2}+\cdots+D^{(\ell)}x^{\alpha_\ell},
\]
with $D^{(k)}\in M_{r\times t}(R)$, and $x^{\alpha_k}:=x_1^{\alpha_{k1}}\cdots x_n^{\alpha_{kn}}$,
where $\alpha_{k1}+\cdots+\alpha_{kn}>0$ for every $1\le k\le \ell$.

From $\overline{D(X)}=\overline{F(X)G(X)-H(X)}=\overline{F(X)}\:
\overline{G(X)}-\overline{H(X)}=\overline{0}$ we conclude that
\[
\overline{D^{(1)}}x^{\alpha_1}+\overline{D^{(2)}}x^{\alpha_1}+
\cdots+\overline{D^{(\ell)}}x^{\alpha_\ell}=\overline{0}.
\]
Then, each entry $d_{ij}^{(k)}$ of the matrix $D^{(k)}$ is such that
$\frac{d_{ij}^{(k)}}{1}=\frac{0}{1}$ in $S^{-1}R$, so we find $s_{ij}^{(k)}\in S$ such that
$s_{ij}^{(k)}d_{ij}^{(k)}=0$. Let $s:= \prod_{i,j,k}s_{ij}^{(k)}$, then $s\in Z(B)$ and
$D^{(k)}s=0$ for each $k=1,\ldots,\ell$. Thus, using Lemma \ref{931g}, we get
\begin{center}
$F(sX)G(sX)-H(sX)=D(sX)=D^{(1)}s\sigma(s)\sigma^2(s)\cdots \sigma^{\alpha_{11}+\cdots
+\alpha_{1n}-1}(s)x^{\alpha_1}+\cdots +D^{(\ell)}s\sigma(s)\sigma^2(s)\cdots
\sigma^{\alpha_{11}+\cdots +\alpha_{ln}-1}(s)x^{\alpha_\ell}=0$.
\end{center}
\end{proof}

\begin{theorem}[\textbf{Vaserstein's theorem}]\label{GVasersteinTeorem}
Let $R$ be a commutative ring and $A:=R[x_1,\dots,x_n;\sigma]$. Then, $F(X)\in M_{r\times s}(A)$ is
equivalent to $F(0)$ if and only if $F(X)$ is locally equivalent to $F(0)$ for every $\mm\in\Max
R$.
\end{theorem}
\begin{proof}
$\Rightarrow)$: Evident.

$\Leftarrow)$: We denote $I$ the set of elements $a\in R$ with the following property:
\begin{center}
Given $f=(f_1,\ldots,f_n),g=(g_1,\ldots,g_2)\in A^{n}$ with $f-g\in a\,A^{n}$, then $F(f)\sim
F(g)$.
\end{center}

$F(f)$ represents the evaluation $x_i\mapsto f_i$, $1\leq i\leq n$, on the matrix $F(X)$. We claim
that $I$ is an ideal of $R$. In fact, let $a,b\in I$ and $f-g\in(a-b)A^{n}$, then $f-g=(a-b)h$,
with $h\in A^{n}$, so $f-(g-bh)=ah\in a\,A^{n}$, and hence $F(f)\sim F(g-bh)$. But $g-(g-bh)=bh\in
b\,A^{n}$, so $F(g-bh)\sim F(g)$, whence, $a-b\in I$. Let $r\in R$, $a\in I$ and $f-g\in
ar\,A^{n}\subseteq a\,A^{n}$, therefore $F(f)\sim F(g)$, and this means that $ar\in I$.

If we show that $I=R$, then for every $f,g\in A^{n}$, $F(f)\sim F(g)$, in particular, if
$f=(x_1,\ldots,x_n)$ and $g=(0,\ldots,0)$, we obtain $F(X)\sim F(0)$. Let $\mm\in\Max R$; there
exists $\overline{G(X)}\in\GL_r(R_{\mm}[x_1,\dots,x_n;\overline{\sigma}])$ and $\overline{H(X)}\in
\GL_s(R_{\mm}[x_1,\dots,x_n;\overline{\sigma}])$ such that
\[
\overline{F(X)}=\overline{G(X)}\, \overline{F(0)}\, \overline{H(X)}.
\]
Introducing the Ore extension $R_{\mm}[x_1,\dots,x_n;y_1,\dots,y_n;\overline{\sigma}]$, i.e.,
$x_ix_j=x_jx_i$, $ x_iy_j=y_jx_i$ and $y_iy_j=y_jy_i$ for $1\leq i,j\leq n$, and $y_i\frac{r}{s}:=
\overline{\sigma}(\frac{r}{s})y_i=\frac{\sigma(r)}{\sigma(s)}y_i$, we obtain, from Proposition
\ref{931g}, that
\begin{equation}\label{matsxy}
\overline{F(X+Y)}=\overline{G(X+Y)}\, \overline{F(0)}\, \overline{H(X+Y)}, \text{ where
}Y:=(y_1,\ldots,y_n).
\end{equation}
Since $\overline{F(0)}=\overline{G(X)}^{-1}\overline{F(X)}\ \overline{H(X)}^{-1}$, we get
\[
\overline{F(X+Y)}=\overline{G(X+Y)}\ \overline{G(X)}^{-1}\overline{F(X)}\ \overline{H(X)}^{-1}\,
\overline{H(X+Y)}.
\]
Denote $G^*:=\overline{G(X+Y)}\ \overline{G(X)}^{-1}$ and $H^*:=
\overline{H(X)}^{-1}\,\overline{H(X+Y)}$. Observe that $G^*$ has the form
\[
\overline{G_0(X)}+\overline{G_1(X)}y^{\alpha_1}+\cdots+ \overline{G_\ell(X)}y^{\alpha_\ell},
\]
with $\overline{G_i(X)}\in M_{r}(R_{\mm}[x_1,\dots,x_n;\overline{\sigma}])$, for every
$i=1,\ldots,\ell$, where $\overline{G_0(X)}$ is the identity matrix, and
$y^{\alpha_i}:=y_1^{\alpha_{i1}}\cdots y_n^{\alpha_{in}}$, for every $1\le i\le \ell$. Moreover,
\begin{equation}\label{cixy}
\overline{G_i(X)}y^{\alpha_i}=\overline{E_0}y^{\alpha_i}+\cdots+\overline{E_{i_j}}
y^{\alpha_i}x^{\beta_{i_j}},\text{ where } \overline{E_k}\in M_r(R_{\mm}),\text{ for } 0\le k\le
i_j.
\end{equation}
Taking a common denominator we find $s'\in S$ and matrices $D_k\in M_r(R)$ such that
\begin{center}
$\overline{E_k}=\frac{D_k}{s'}=\frac{D_k\sigma(s')\sigma^2(s')\cdots \sigma^{\alpha_{i1}+\cdots
+\alpha_{in}-1}(s')}{s'\sigma(s')\sigma^2(s')\cdots \sigma^{\alpha_{i1}+\cdots
+\alpha_{in}-1}(s')}$,
\end{center}
so we can assume that $\overline{E_k}=\frac{D_k}{s'\sigma(s')\sigma^2(s')\cdots
\sigma^{\alpha_{i1}+\cdots +\alpha_{in}-1}(s')}$

Hence, replacing $Y$ by $s'Y$ we get that $\overline{G(X+s'Y)}\: \overline{G(X)}^{-1}$ is the image
of a matrix with entries over $R[x_1,\dots,x_n;y_1,\dots,y_n;\sigma]$. In a similar way we can do
with $\overline{H(X)}^{-1}\, \overline{H(X+Y)}$. Thus, we can suppose that
\[
\overline{G(X+s'Y)}\:\overline{G(X)}^{-1}\quad\text{ and }\quad
\overline{H(X)}^{-1}\,\overline{H(X+s'Y)}
\]
are images of invertible matrices
\[
\Gamma(X,Y)\quad\text{ and }\quad\Delta(X,Y),
\]
respectively, with entries in $R[x_1,\dots,x_n;y_1,\dots,y_n;\sigma]$, where $\Gamma(X,0)$ and
$\Delta(X,0)$ are identities matrices.

On $R_{\mm}[x_1,\dots,x_n;y_1,\dots,y_n;\overline{\sigma}]$, we have the equation
\[
\overline{F(X+s'Y)}=\overline{G(X+s'Y)}\:\overline{G(X)}^{-1}
\overline{F(X)}\:\overline{H(X)}^{-1}\, \overline{H(X+s'Y)},
\]
and on $R[x_1,\dots,x_n;\sigma]$,
\[
F(X)=\Gamma(X,0)F(X)\Delta(X,0).
\]
Taking $B:=R[x_1,\dots,x_n;\sigma]$ in Lemma \ref{lema1720}, there exists $s''\in R-\mm$ such that
for $s:=s's''$, we have the equation
\[
F(X+sY)=\Gamma(X,s''Y)F(X)\Delta(X,s''Y)
\]
in the Ore extension $R[x_1,\dots,x_n;y_1,\dots,y_n;\sigma]$. Now if $f,g,h\in A^{n}$ are such that
$f-g=sh$, we have
\[
F(f)=F(g+sh)=\Gamma(g,s''h)F(g)\Delta(g,s''h),
\]
where $\Gamma(g,s''h)$ and $\Delta(g,s''h)$ are invertible; then $F(f)\sim F(g)$ and so $s\in I$.
We have showed that for every $\mm\in\Max R$ there exists $s\in I$ with $s\notin\mm$, i.e., $I=R$.

\end{proof}

\section{Quillen's patching theorem}

Now we will study another classical result of commutative algebra for the Ore extensions of type
$A:=R[x_1,\dots,x_n;\sigma]$, with $R$ commutative: the famous Quillen's patching theorem. For this
we will adapt the method studied in \cite{Kunz}. Let $B$ be a ring and consider two exact sequences
of $B$-modules
\begin{align*}
&0\to K_1\xto{\beta_1}F_1\xto{\alpha_1}M_1\to 0\\
&0\to K_2\xto{\beta_2}F_2\xto{\alpha_2}M_2\to 0.
\end{align*}
where $F_1, F_2$ are free.

\begin{proposition}\label{prop1}
If $i:M_1\too M_2$ is an isomorphism, then there exists $\alpha\in\Aut(F_1\oplus F_2)$ such that
following diagram
\begin{equation}\label{dia1}
\begin{diagram}
\node{F_1\oplus F_2}\arrow{e,t}{(\alpha_1,0)}\arrow{s,l}{\alpha}%
\node{M_1}\arrow{s,r}{i}\\
\node{F_1\oplus F_2}\arrow{e,b}{(0,\alpha_2)}\node{M_2}
\end{diagram}
\end{equation}
is commutative. Identifying $K_j$ with $\beta_j(K_j)\subseteq F_j$
$(j=1,2)$, $\alpha(K_1\oplus F_2)=F_1\oplus K_2$.%
\end{proposition}
\begin{proof}
The proof is exactly as in \cite{Kunz} and it is not necessary to assume that $B$ is commutative.
\end{proof}

\begin{corollary}\label{coroLocalQuillen}
Let $B$ be a ring and consider two exact sequences of $B$-modules
\begin{equation}\label{sucesiones}
F'_j\xrightarrow{\beta_j}F_j\xrightarrow{\alpha_j}M_j\too 0\qquad (j=1,2),
\end{equation}
where $F_j,F'_j$ are free $B$-modules. Then, $M_1\cong M_2$ if and only if there exist
$\alpha\in\Aut(F_1\oplus F_2)$ and $\beta\in\Aut(F'_1\oplus F_2\oplus F_1\oplus F'_2)$ such that
the following diagram
\begin{equation}\label{dia2}
\begin{diagram}
\node{F'_1\oplus
F_2\oplus F_1\oplus F'_2\quad}\arrow{e,t}{(\beta_1\oplus i_{F_2},0)}\arrow{s,l}{\beta}%
\node{\quad F_1\oplus F_2}\arrow{s,r}{\alpha}\\
\node{F'_1\oplus F_2\oplus F_1\oplus F'_2\quad}\arrow{e,b}{(0,i_{F_1}\oplus\beta_2)}\node{\quad
F_1\oplus F_2}
\end{diagram}
\end{equation}
commutes.
\end{corollary}
\begin{proof}
See \cite{Kunz}.
\end{proof}

\begin{remark}
Now we can consider that $M_1$ and $M_2$ are finitely presented $B$-modules
\begin{align*}
&F_1'\xto{\beta_1}F_1\xto{\alpha_1}M_1\to 0\\
&F_2'\xto{\beta_2}F_2\xto{\alpha_2}M_2\to 0;
\end{align*}
with respect to the canonical bases, $\beta_1$ is given by a matrix $B_1\in M_{m'\times m}(B)$ and
$\beta_2$ by $B_2\in M_{n'\times n}(B)$. The matrices that represent the homomorphisms in the rows
of (\ref{dia2}) are given by
\begin{equation}\label{equ3}
\begin{bmatrix}
\begin{tabular}{c|c}
  $B_1$ & $0$ \\
  \hline
  $0$ & $I_n$ \\
  \hline
\end{tabular}\\
0
\end{bmatrix}
\qquad%%
\text{and}%%
\qquad%%
\begin{bmatrix}
0\\
\begin{tabular}{c|c}
  \hline
  $I_m$ & $0$ \\
  \hline
  $0$ & $B_2$ \\
\end{tabular}
\end{bmatrix},
\end{equation}
and the matrices of homomorphisms in the columns are in $M_{r}(B)$ and $M_s(B)$, where
$r:=m+n+m'+n'$ and $s:=m+n$. Therefore, Corollary \ref{coroLocalQuillen} says that the modules
$M_1$ and $M_2$ are isomorphic if and only if the matrices in (\ref{equ3}) are equivalent.
\end{remark}

Note that  $A_{\mm}$ is a right $A$-module and if $M$ is a left $A$-module, then we denote
\begin{center}
$M_{\mm}:=A_{\mm}\otimes_{A}M=R_{\mm}[x_1,\dots,x_n;\overline{\sigma}]\otimes_{A}M$.
\end{center}
If $N$ is a right $R$-module, then we denote
\begin{center}
$N[x_1,\dots,x_n;\sigma]:=N\otimes_R A$.
\end{center}

\begin{theorem}[\textbf{Quillen's patching theorem}]\label{tpq}
Let $R$ be a commutative ring and $A:=R[x_1,\dots,x_n;\sigma]$. Let $M$ be a finitely presented
$A$-module. $M$ is extended from $R$ if and only if $M_{\mm}$ is extended from $R_{\mm}$, for every
$\mm\in\Max(R)$.
\end{theorem}
\begin{proof}
There exists an exact sequence of $A$-modules
\begin{equation}\label{suc_1}
A^p\xrightarrow{\beta_1}A^q\xrightarrow{\alpha_1}M\too 0.
\end{equation}
Tensoring by $A/\langle X\rangle$ we obtain the exact sequence of $R$-modules
\begin{equation}\label{suc_2}
R^p\xrightarrow{\ol{\beta}_1}R^q\xrightarrow{\ol{\alpha}_1}M/\langle X\rangle M\too0.
\end{equation}
If $B\in M_{p\times q}(A)$ is the matrix of $\beta_1$ with respect to the canonical bases, then
$B(0)$ is the matrix of $\ol{\beta}_1$. From (\ref{suc_2}) we get an exact sequence of $A$-modules,
where $N:=M/\langle X\rangle M$: {\footnotesize
\begin{equation}\label{suc_3}
R^p[x_1,\dots,x_n;\sigma]\xrightarrow{\ol{\beta}_1[X]}R^q[x_1,\dots,x_n;\sigma]
\xrightarrow{\ol{\alpha}_1[X]} N[x_1,\dots,x_n;\sigma]\to 0.
\end{equation}}
Note that $\ol{\beta}_1[X]:=\overline{\beta_1}\otimes i_A$ and
$\ol{\alpha}_1[X]:=\overline{\alpha_1}\otimes i_A$. The sequence (\ref{suc_3}) can be identified
with the exact sequence of $A$-modules
\begin{equation}\label{suc_4}
A^p\xrightarrow{\beta_2}A^q\xrightarrow{\alpha_2} N[x_1,\dots,x_n;\sigma]\too 0,
\end{equation}
where $\beta_2$ is described by the matrix $B(0)$. From Corollary \ref{coroLocalQuillen}, we have
$M\cong N[x_1,\dots,x_n;\sigma]$ if and only if the $2(p+q)\times(2q)$-matrices
\[
F:=
\begin{bmatrix}
\begin{tabular}{c|c}
  $B$ & $0$ \\
  \hline
  $0$ & $I_q$ \\
  \hline
\end{tabular}\\
0
\end{bmatrix}
\qquad%%
\text{and}%%
\qquad%%
G:=\begin{bmatrix}
0\\
\begin{tabular}{c|c}
  \hline
  $I_q$ & $0$ \\
  \hline
  $0$ & $B(0)$ \\
\end{tabular}
\end{bmatrix}
\]
are equivalent. Note that $G$ is equivalent to $F(0)$ (permuting rows and columns). Therefore, by
Theorem \ref{GVasersteinTeorem}, $M\cong N[x_1,\dots,x_n;\sigma]$ if and only if $F$ and $F(0)$ are
locally equivalent for every $\mm\in\Max(R)$. Since the exact sequences (\ref{suc_1}) and
(\ref{suc_4}) are consistent with respect to the localization by $\mm$, we get that $M$ is extended
from $R$ if and only if $M_{\mm}$ is extended from $R_{\mm}$, for every $ \mm\in\Max(R)$.
\end{proof}

\begin{corollary}\label{945}
Let $R$ be a commutative ring and $A:=R[x_1,\dots,x_n;\sigma]$. If for each $\mm\in\Max R$,
$A_{\mm}$ is $\mathcal{E}$ with respect to $R_\mm$, then $A$ is $\mathcal{E}$.
\end{corollary}
\begin{proof}
Let $M\in\pry(A)$, then $M_{\mm}\in\pry(A_{\mm})$, and hence, by the hypothesis, $M_{\mm}$ is
extended from $R_{\mm}$, for every $\mm\in\Max R$. Using Theorem \ref{tpq} (note that $M$ is
finitely presented as $A$-module), $M$ is extended from $R$, and hence, $A$ is $\mathcal{E}$.
\end{proof}

\section{Quillen-Suslin theorem}

This last section concerns with Quillen-Suslin's theorem. Here $R$ is non-commutative but some
other extra conditions on it are assumed as well as over $\sigma$.

\begin{theorem}[Horrocks' theorem]
Let $R$ be a left regular domain and $Z$ its center. Suppose that $Z$ is Noetherian, $R$ is
finitely generated over $Z$ and  $\sigma$ is an automorphism of $R$ of finite order $d$, with $d$
invertible in $Z$. Suppose that $P\in \pry(A)$ is stably extended from $R$ and the rank of $P$ is
at least 2. Then $P$ is extended from $R$.
\end{theorem}
\begin{proof}
Firstly we recall that the rank of $P$ means the maximal number of $R$-independent elements of $P$,
and $P$ is stably extended from $R$ is there exists $m\geq 0$ such that $P\oplus A^m$ is extended
from $R$.

Denote by $d$  the order of  $\sigma$. The automorphism $\sigma$ is mapping $Z$ onto itself.   Let
$Z^{\sigma}$ be the subalgebra in $Z$  of invariants of  $\sigma$. By Noether's theorem
$Z^{\sigma}$ is Noetherian  and $Z$ is a finitely generated  $Z^{\sigma}$-module. We claim that
$Z^{\sigma}[x_1^d,\ldots, x_n^d]$ is a central subalgebra in $A$ and $A$ is a finitely generated
left $Z^{\sigma}[x_1^d,\ldots, x_n^d]$-module. In fact monomials in $x_1,\ldots,x_n$ commute. Since
$x_ir =\sigma(r)x_i$ for any $i$ and for any $r\in R$, then
 $x_i^d r=\sigma^d(r)x_i^d=  rx^d$. Hence each element $x_i^d$ is central. Now  if $r\in Z^{\sigma}$ then $r$
 commutes with each element of $R$ and with each variable $x_i$. Since $R$ is a finitely  generated $Z$-module then by
 Noether's theorem $R$ is a finitely generated $Z^{\sigma}$-module. So the claim is proved. Consider $A$ as a
 graded ring $A=\oplus_n A_n$ where $A_n$ is the span of all monomials  of a total
degree $n$. In particular $A_0=R$. Let $V$ be the graded ideal in $D=Z^{\sigma}[x_1^d,\ldots,
x_n^d]$ considered in \cite[Theorem 5.32]{Artamonov2}. Let $\wp $ be a maximal graded ideal in $D$
and $A_{\wp}^+$ the localization considered in  \cite[Definition 5.30]{Artamonov2}.  As it was
shown in  \cite[Corollaries 5.36]{Artamonov2} either $V$ contains all $x_1^d,\ldots,x_n^d$  or the
ring $A_{\wp}^+$ is a skew polynomial extension $A^+_{\wp}(0)[x_i,\alpha]$ for some $x_i$. In the
first case if each $x_i^d\in V$ then $P_{x_i^d}$ is extended from $R$ and $x_i^d$ is monic. In the
second case using the restriction on the rank of $P$   we can find a monic polynomial  $f$ in $x_i$
such that $P_f$ is extended from $R$. So in bother cases by \cite[Proposition 5.35]{Artamonov2} the
module $P$ is extended from $R$.
\end{proof}
\begin{theorem}[Quillen-Suslin theorem]\label{quillensuslinforore}
Let $K$ be a field and $A:=K[x_1,\dots,x_n;\sigma]$, with $\sigma$ bijective and having   finite
order. Then $A$ is $\mathcal{PF}$.
\end{theorem}
\begin{proof}
Apply previous theorem with $R=Z=K$.
\end{proof}

%-----------------------------------

%\begin{center}
%\textbf{Acknowledgements}
%\end{center}
%The authors are grateful to the editors and the referee for valuable suggestions and corrections.

%%%%%%%%%%%%%%%%%%%%%%%%%%%%%%%%%%%%%%%%%%%%%%%%%
%%%%%%%%%%%%%%%%%%%%%%%%%%%%%%%%%%%%%%%%%%%%%%%%%

\end{document}